\documentclass[reqno, 11pt, a4paper]{amsart}
\usepackage{amsmath,amsfonts,amsthm,amssymb,amscd}
\usepackage{hyperref}
\usepackage[alphabetic]{amsrefs}
\usepackage{tikz}
\usepackage{graphicx}
\usepackage{xcolor}

\usepackage[hmargin=34mm,vmargin=41.6mm,footskip=5ex]{geometry}

\makeatletter
\def\ps@headings{\ps@empty
  \def\@evenhead{%
    \setTrue{runhead}%
    \normalfont\scriptsize
    \hfil
    \def\thanks{\protect\thanks@warning}%
    \leftmark{}{}\hfil}%
  \def\@oddhead{%
    \setTrue{runhead}%
    \normalfont\scriptsize \hfil
    \def\thanks{\protect\thanks@warning}%
    \rightmark{}{}\hfil}%
  \def\@oddfoot{\normalfont\scriptsize \hfil\thepage\hfil}%
  \let\@evenfoot\@oddfoot
  \let\@mkboth\markboth
}
\makeatother

\numberwithin{equation}{section}	

\newtheorem{lem}{Lemma}[section]

\newtheorem{tw}[lem]{Theorem}
\newtheorem*{wtw}{Theorem}

\newtheorem{prop}[lem]{Proposition}

\theoremstyle{definition}

\theoremstyle{remark}
\newtheorem{rem}[lem]{Remark}

\begin{document}

\title[Solvable subgroup theorem]{Solvable Subgroup Theorem for simplicial nonpositive curvature}
\author[T. Prytu{\l}a]{Tomasz Prytu\l a}
\date{\today}

\address{School of Mathematics, University of Southampton, Southampton SO17 1BJ, UK}
\email{t.p.prytula@soton.ac.uk}

\begin{abstract}
Given a group $G$ with bounded torsion that acts properly on a systolic complex, we show that every solvable subgroup of $G$ is finitely generated and virtually abelian of rank at most $2$. In particular this gives a new proof of the above theorem for systolic groups.
The main tools used in the proof are the Product Decomposition Theorem and the Flat Torus Theorem.
\end{abstract}

\thanks{The author was supported by the EPSRC First Grant EP/N033787/1.}
\maketitle

\section{Introduction}
\label{sec:intro}

For a group $G$ the \emph{Solvable Subgroup Theorem} states that any solvable subgroup of $G$ is finitely generated and virtually abelian. This theorem holds for several classes of groups, eg.,\ fundamental groups of nonpositively curved Riemannian manifolds \cite{GrWo,LaYa} and more generally $\mathrm{CAT}(0)$ groups \cite{BH}, translation discrete groups of finite virtual cohomological dimension \cite{Con}, biautomatic groups whose abelian subgroups are finitely generated \cite{GeSh}. In this note we show this theorem for groups acting on systolic complexes. 

\begin{wtw}[Solvable Subgroup Theorem]\label{tw:main}Let $G$ be a group acting properly on a uniformly locally finite systolic complex and suppose that there is a bound on the order of finite subgroups of $G$. Then every solvable subgroup of $G$ is finitely generated virtually abelian of rank at most $2$.
\end{wtw}

A systolic complex is a simply connected simplicial complex that is flag and that has no induced cycles of length $4$ or $5$. The latter condition is sometimes referred to as ``simplicial nonpositive curvature''. In our proof we make no use of the definition and only use properties of groups acting on systolic complexes. For a detailed treatment of systolic complexes we refer the reader to \cite{JS2}.

The most natural examples of groups satisfying assumptions of the Theorem are \emph{systolic groups}, i.e., groups acting properly and cocompactly on systolic complexes. For such groups the Theorem follows from their biautomaticity and the fact that their abelian subgroups are finitely generated (see \cite[Theorem~2.2]{HuOs} for a short account of the proof).

In general, the group appearing in the Theorem does not have to be (bi-)automatic and thus a different approach is needed. Our proof is based on the Flat Torus Theorem and the Product Decomposition Theorem and is similar to the one in \cite{BH} for $\mathrm{CAT}(0)$ groups. The main difference is that in our case the above two theorems determine the structure of normalisers of abelian subgroups of $G$, and thus one can proceed more directly than in the $\mathrm{CAT}(0)$ case.

In Section~\ref{sec:translation} we present another approach to the Theorem, via theory of translation discrete groups. We show that $G$ is translation discrete and then using results of \cite{Con} we give an alternative proof of the Theorem. However, this approach seems  less straightforward.

\section{Proof of the Theorem}
Throughout this section we assume that $G$ is as in the Theorem, and let $X$ be a systolic complex on which $G$ acts properly. We need the following two lemmas. The first one sums up the consequences of the Product Decomposition Theorem and the Flat Torus Theorem that are needed in the proof. For the actual statements of these theorems see \cite[Theorem~A]{OsaPry} and \cite[Theorem 6.1]{E1} respectively.

\begin{lem}\label{lem:flattorproduct}Let $A$ be a finitely generated free abelian subgroup of $G$. Then:
\begin{enumerate}
\item \label{item:rank2} the rank of $A$ is at most $2$,
\item \label{item:norm2} if $A \cong \mathbb{Z}^2$ then $A$ has finite index in its normaliser $N_G(A)$,
\item \label{item:norm1}if $A \cong \mathbb{Z}$ then any finitely generated subgroup of the normaliser $N_G(A)$ that contains $A$ is virtually $F_n \times A'$, where $F_n$ is the free group on $n$ generators for some $n \geqslant 0$ and $A'$ is a finite-index subgroup of $A$. 
\end{enumerate}
\end{lem}

\begin{proof}
(1) Follows immediately from \cite[Theorem~6.1(1)]{E1}.

\noindent(2) The proof is essentially the same as the proof of \cite[Lemma~5.10(NM2)]{OsaPry}. We will give a sketch here. Consider a subcomplex $\mathrm{Min}(A) \subset X$, which consists of all vertices of $X$ that are moved the minimal combinatorial distance by all elements of $A$. By the Flat Torus Theorem (\cite[Theorem~6.1(3)]{E1}) we have that $\mathrm{Min}(A)$ is non-empty, $A$--invariant and the action of $A$ on $\mathrm{Min}(A)$ is proper and cocompact. Moreover, it follows from the definition that $\mathrm{Min}(A)$ is also $N_G(A)$--invariant. The action of $N_G(A)$ on $\mathrm{Min}(A)$ is proper (since it is the restriction of a proper action of $G$) and it is cocompact since $A \subset N_G(A)$. This implies that the inclusion $A \subset N_G(A)$ is a quasi-isometry and thus $[N_G(A) : A]$ is finite.   

\noindent(3) In \cite[Proposition~5.7]{OsaPry} the analogous claim is proven for finitely generated subgroups of the centraliser $C_G(A)$. One easily verifies that the exact same proof works if one replaces $C_G(A)$ by $N_G(A)$.
\end{proof}

The second lemma is a mild strengthening of \cite[Proposition~5.14]{OsaPry}.
\begin{lem}[Ascending chain condition]\label{lem:ascchain} 
For any chain
\[A_1 \subset A_2 \subset A_3 \subset \ldots\] of virtually abelian subgroups of $G$ there exists $n>0$ such that $A_i =A_{i+1}$ for all $i\geqslant n$. In particular, any virtually abelian subgroup of $G$ is finitely generated.
\end{lem}

\begin{proof} 
First note that if there were a strictly ascending chain of virtually abelian subgroups of $G$ then there would be a strictly ascending chain of finitely generated virtually abelian subgroups. Thus we can assume that every $A_i$ is finitely generated, and therefore by Lemma~\ref{lem:flattorproduct}(\ref{item:rank2}) it is virtually $\mathbb{Z}^{n_i}$ where $n_i\leqslant 2$. Without loss of generality we can assume that every $A_i$ is virtually $\mathbb{Z}^n$ for a fixed $n\leqslant 2$. We consider the three possibilities for $n$. If $n=0$ then all $A_i$'s are finite and the claim follows from the fact that the order of finite subgroups of $G$ is bounded. For $n=2$ the claim is proven in \cite[Lemma~5.10(NM1)]{OsaPry}. It remains to prove the case where $n=1$. 

Let $a \in A_1$ be an infinite order element. We will show that the index $[A_i : \langle a \rangle]$ is uniformly bounded. Since each $A_i$ is virtually $\mathbb{Z}$, by \cite[Proposition~4]{LePi} it contains a maximal normal finite subgroup $N_i$ such that $A_i/N_i$ is either infinite cyclic or infinite dihedral. In either case, the group $A_i/N_i$ contains an infinite order element $b_i$ such that $[A_i/N_i : \langle b_i \rangle] \leqslant 2$. Since $a$ is of infinite order, it injects into the quotient $A_i/N_i$ and thus $a=b_i^{k_i}$ for some $k_{i} \in \mathbb{Z}$. 

In Proposition~\ref{prop:sys-trans-discrete} we show that $G$ is \emph{translation discrete}. It is straightforward to check that for such groups, for a fixed $a \in G$ there is only finitely many $k \in \mathbb{Z}$ such that $a=b^k$ for some $b \in G$. Thus there exists $K \in \mathbb{N}$ such that for all $i$ we have $|k_i| \leqslant K$. Consequently, we have $[\langle b_i \rangle : \langle a \rangle]\leqslant K$ and hence $[A_i/N_i : \langle a \rangle]\leqslant 2K$. Finally, since the order of finite subgroups of $G$ is bounded by some  $M >0$, we obtain that $[A_i : \langle a \rangle]\leqslant 2K|N_i| \leqslant 2KM$. \end{proof}

For a group $H$ let $[H,H]$ denote the commutator subgroup of $H$. Put $H^{(1)} = [H,H]$ and define recursively $H^{(n)}= [H^{(n-1)},H^{(n-1)}] $. Recall that $H$ is \emph{solvable} if $H^{(n)}=\{e\}$ for some $n\geqslant 1$. The smallest number $n$ with this property is called the \emph{solvability rank} of $H$. We are ready now to prove the Theorem. 

\begin{proof}[Proof of the Theorem] 
In the light of Lemma~\ref{lem:ascchain} it suffices to consider finitely generated solvable subgroups of $G$. Let $H$ be such a subgroup. We proceed by induction on the solvability rank of $H$. If $H^{(1)}=\{e\}$ then $H$ is abelian and the claim follows by Lemma~\ref{lem:flattorproduct}(\ref{item:rank2}). Now assume that $H^{(1)}$ is finitely generated virtually abelian of rank at most $2$. Let $A \subset H^{(1)}$ be a free abelian subgroup of finite index. We consider the three possible cases for the rank of $A$:
\begin{enumerate}

\item $rk(A)=0.$ This means that $H^{(1)}$ is finite and thus by \cite[Lemma II.7.9]{BH} the group $H$ contains an abelian subgroup of finite index. The claim follows from Lemma~\ref{lem:flattorproduct}(\ref{item:rank2}).

\item $rk(A)=1.$ Define $A' \subset H^{(1)}$ to be the intersection of all cyclic subgroups of $H^{(1)}$ of index $[H^{(1)}: A]$ (there are only finitely many such subgroups and thus the index $[H^{(1)}:A']$ is finite). By construction $A'$ is a characteristic subgroup of $H^{(1)}$ and hence a normal subgroup of $H$, and therefore we have $H \subset N_G(A')$. Since $H$ is finitely generated, by Lemma~\ref{lem:flattorproduct}(\ref{item:norm1}) we conclude that $H$ is virtually $F_n \times A''$ where $A''$ is a finite-index subgroup of $A'$. Because $H$ is solvable we necessarily have $n \leqslant 1$ for otherwise $H$ would contain a non-abelian free group. This finishes the proof since for $n \leqslant 1$ the product $F_n \times A''$ is isomorphic to either $\mathbb{Z}$ or $\mathbb{Z}^2$.
\item $rk(A)=2.$  Proceeding as in the previous case we can find a finite-index subgroup $A' \subset A$ which is characteristic in $H^{(1)}$ and thus normal in $H$. Therefore $H\subset N_G(A')$ and since by Lemma~\ref{lem:flattorproduct}(\ref{item:norm2}) the index $[N_G(A'): A']$ is finite, it follows that the index $[H:A']$ is finite as well. \qedhere
\end{enumerate} 
\end{proof}

We conclude this section with a generalisation of the Theorem to the case of \emph{elementary amenable} subgroups. The class of elementary amenable groups is obtained from finite groups and $\mathbb{Z}$ by taking extensions, increasing unions, subgroups and quotients. In particular, it contains all solvable groups. We refer the reader to \cite{Hillm,HiLi} for a detailed definition. 

\begin{prop}\label{prop:elementamen}Let $H$ be an elementary amenable subgroup of $G$. Then $H$ is virtually solvable, and thus finitely generated virtually abelian of rank at most $2$.
\end{prop}

\begin{proof}We would like to apply \cite[Theorem on page 238]{HiLi} which requires the Hirsch length $h(H)$ of $H$ to be finite. This is the case by the following argument. By \cite[Lemma~2]{Hillm} we have $h(H) \leqslant \mathrm{cd}_{\mathbb{Q}}{H}$, where the latter denotes the rational cohomological dimension of $H$. The group $H$ acts properly on a finite dimensional systolic complex $X$, which is contractible by \cite[Theorem~4.1(1)]{JS2}. It is a standard fact that in this case we have $\mathrm{cd}_{\mathbb{Q}}{H} \leqslant \mathrm{dim}X$ (see \cite[Lemma~3.3]{Pet} for a proof). 

Now by \cite[Theorem on page 238]{HiLi} there is a short exact sequence \[0 \to N \to H \to H/N \to 0\] such that $N$ is locally finite and $H/N$ is virtually solvable. Since the order of finite subgroups of $H$ is bounded, the group $N$ must be finite. This implies that $H$ is itself virtually solvable. \end{proof}

\section{Translation discreteness of systolic groups}\label{sec:translation}
We begin by recalling the definition of translation discrete groups. A \emph{semimetric} on a group $G$ is a function $d \colon G \times G \to \mathbb{R}_{+}$ which is symmetric, satisfies the triangle inequality and for any $g, h_1, h_2 \in G$ we have $d(h_1, h_2)=d(gh_1, gh_2)$. The associated seminorm of an element $g$ is given by $\|g\|=d(e,g)$. The \emph{translation number} of $g$ is defined as
\[\tau(g)=\underset{n \to \infty}{\mathrm{lim\, inf}} \frac{ \|g^n\|}{n}.\]

We say that semimetric $d$ on $G$ is \emph{translation discrete} if the set of translation numbers of infinite order elements of $G$ is bounded away from $0$. A group $G$ is translation discrete if it admits a translation discrete semimetric.

\begin{prop}\label{prop:sys-trans-discrete}Let $G$ be as in the Theorem. Then $G$ is translation discrete.
\end{prop}

\begin{proof}
Let $X$ be a systolic complex on which $G$ acts properly and let $d$ denote the edge-path metric on the $1$--skeleton of $X$. Note that since $G$ acts simplicially on $X$, it acts by isometries with respect to the above metric on $X^{(1)}$. Pick a vertex $x_0 \in X$ and define a semimetric $\tilde{d}$ on $G$ by setting $\tilde{d}(g_1,g_2)= d(g_1x_0, g_2x_0)$. Then the associated seminorm is given by $\|g\|=d(x_0, gx_0)$.

Now suppose $h$ is an infinite order element of $G$. The key idea is that $h$ has a \emph{thick axis} in $X$ which enables us to calculate the norm $\|h^n\|$. More precisely, by \cite[Example~1.2 and Theorem~1.3]{E2} there is an $h$--invariant subcomplex $A_h \subset X$ which is at Hausdorff distance at most $1$ from an $X^{(1)}$--geodesic line, and such that for any vertex $x \in A_h$  we have \[d(x, h^nx)=\lfloor\frac{nm}{k}\rfloor,\] where $m\geqslant 1$ and $\mathrm{dim}X \geqslant k \geqslant 1$ are integers depending only on $h$. Pick a vertex $x \in A_h$ and let $K=d(x_0,x)$. By the triangle inequality we obtain:
\[
d(x_0, h^nx_0) \geqslant d(x,h^nx)-d(x_0,x)-d(h^nx_0,h^nx) =d(x,h^nx)-2K= \lfloor\frac{nm}{k}\rfloor-2K.
\]

Now we can estimate the translation number of $h$. We have

\[
\frac{ \|h^n\|}{n}= \frac{d(x_0, h^nx_0)}{n} \geqslant \frac{\lfloor\frac{nm}{k}\rfloor - 2K}{n} \geqslant \frac{ \frac{nm}{k} - 1 - 2K}{n}= \frac{m}{k} - \frac{1+2K}{n}\] 
and thus
\[\tau(h)=\underset{n \to \infty}{\mathrm{lim\, inf}} \frac{ \|h^n\|}{n} \geqslant \underset{n \to \infty}{\mathrm{lim\, inf}}\bigg ( \frac{m}{k} - \frac{1+2K}{n} \bigg)=  \frac{m}{k} \geqslant  \frac{1}{\mathrm{dim}X}. \qedhere\]

\end{proof}

 For translation discrete groups we have the following version of the Solvable Subgroup Theorem.

\begin{tw}{\cite[Theorem~3.4]{Con}}\label{tw:connsolv}
A solvable subgroup of finite virtual cohomological dimension in a translation discrete group is virtually $\mathbb{Z}^n$.
\end{tw}

We will now sketch how our Theorem can be derived from Theorem~\ref{tw:connsolv}.

\begin{proof}[Alternative proof of the Theorem]
By Lemma~\ref{lem:ascchain} it is enough to consider finitely generated solvable subgroups of $G$. Suppose $H$ is such a subgroup. Again by Lemma~\ref{lem:ascchain} we obtain that all abelian subgroups of $H$ are finitely generated. Then a theorem of Maltsev \cite[Theorem 2 on page 25]{Seg} implies that $H$ is polycyclic, and thus virtually torsion-free.

It remains to show that $H$ has finite virtual cohomological dimension. The group $H$ acts properly on a systolic complex $X$ and thus any torsion-free subgroup of $H$ acts freely on $X$. Since the complex $X$ is contractible \cite[Theorem~4.1(1)]{JS2}, 
we have $\mathrm{vcd}H \leqslant \mathrm{dim}X < \infty$. Now the claim follows from Theorem~\ref{tw:connsolv} and Lemma~\ref{lem:flattorproduct}(\ref{item:rank2}).
 \end{proof}

\begin{rem}We would like to point out that the two presented approaches are closely related. Namely, one of the key notions used in the proofs of the Flat Torus Theorem and the Product Decomposition Theorem is the \emph{translation length} of an infinite order element, which is an ``unstable'' analogue of the translation number. This concept is implicitly used in the proof of Proposition~\ref{prop:sys-trans-discrete}.

\end{rem}

\begin{rem}In the light of Proposition~\ref{prop:elementamen} one could ask whether every amenable subgroup of $G$ is virtually abelian. In fact, a much stronger statement could be possibly true. Namely, the \emph{Tits alternative} asserts that any finitely generated subgroup of $G$ is either virtually abelian or it contains a non-abelian free group. Since amenable (and in particular solvable) groups do not contain free groups, the Theorem would immediately follow from the Tits alternative. 
\end{rem}



\begin{bibdiv}

\begin{biblist}


\bib{BH}{book}{
  author={Bridson, Martin R.},
  author={Haefliger, Andr{\'e}},
  title={Metric spaces of non-positive curvature},
  series={Grundlehren der Mathematischen Wissenschaften [Fundamental
    Principles of Mathematical Sciences]},
  volume={319},
  publisher={Springer-Verlag, Berlin},
  date={1999},
  pages={xxii+643},
  isbn={3-540-64324-9},
  review={\MR{1744486}},
  doi={10.1007/978-3-662-12494-9},
}




\bib{Con}{article}{
   author={Conner, Gregory R.},
   title={Discreteness properties of translation numbers in solvable groups},
   journal={J. Group Theory},
   volume={3},
   date={2000},
   number={1},
   pages={77--94},
   issn={1433-5883},
   review={\MR{1736519}},
   doi={10.1515/jgth.2000.007},
}





\bib{E1}{article}{
   author={Elsner, Tomasz},
   title={Flats and the flat torus theorem in systolic spaces},
   journal={Geom. Topol.},
   volume={13},
   date={2009},
   number={2},
   pages={661--698},
   issn={1465-3060},
   review={\MR{2469526 (2009m:20065)}},
   doi={10.2140/gt.2009.13.661},
}


\bib{E2}{article}{
   author={Elsner, Tomasz},
   title={Isometries of systolic spaces},
   journal={Fund. Math.},
   volume={204},
   date={2009},
   number={1},
   pages={39--55},
   issn={0016-2736},
   review={\MR{2507689 (2010g:51005)}},
   doi={10.4064/fm204-1-3},
}

\bib{GeSh}{article}{
   author={Gersten, S. M.},
   author={Short, H. B.},
   title={Rational subgroups of biautomatic groups},
   journal={Ann. of Math. (2)},
   volume={134},
   date={1991},
   number={1},
   pages={125--158},
   issn={0003-486X},
   review={\MR{1114609 (92g:20092)}},
   doi={10.2307/2944334},
}


\bib{GrWo}{article}{
   author={Gromoll, Detlef},
   author={Wolf, Joseph A.},
   title={Some relations between the metric structure and the algebraic
   structure of the fundamental group in manifolds of nonpositive curvature},
   journal={Bull. Amer. Math. Soc.},
   volume={77},
   date={1971},
   pages={545--552},
   issn={0002-9904},
   review={\MR{0281122}},
   doi={10.1090/S0002-9904-1971-12747-7},
}

\bib{Hillm}{article}{
   author={Hillman, Jonathan A.},
   title={Elementary amenable groups and $4$-manifolds with Euler
   characteristic $0$},
   journal={J. Austral. Math. Soc. Ser. A},
   volume={50},
   date={1991},
   number={1},
   pages={160--170},
   issn={0263-6115},
   review={\MR{1094067}},
}



\bib{HiLi}{article}{
   author={Hillman, J. A.},
   author={Linnell, P. A.},
   title={Elementary amenable groups of finite Hirsch length are
   locally-finite by virtually-solvable},
   journal={J. Austral. Math. Soc. Ser. A},
   volume={52},
   date={1992},
   number={2},
   pages={237--241},
   issn={0263-6115},
   review={\MR{1143191}},
}



\bib{HuOs}{article}{
    title     ={Large-type Artin groups are systolic},
    author    ={Huang, Jingyin},
    author    ={Osajda, Damian},
    status    ={preprint},
    date      ={2017},
    eprint    ={https://arxiv.org/abs/1706.05473},
}


\bib{JS2}{article}{
   author={Januszkiewicz, Tadeusz},
   author={{\'S}wi{\c{a}}tkowski, Jacek},
   title={Simplicial nonpositive curvature},
   journal={Publ. Math. Inst. Hautes \'Etudes Sci.},
   number={104},
   date={2006},
   pages={1--85},
   issn={0073-8301},
   review={\MR{2264834 (2007j:53044)}},
}


\bib{LePi}{article}{
   author={Juan-Pineda, Daniel},
   author={Leary, Ian J.},
   title={On classifying spaces for the family of virtually cyclic
   subgroups},
   conference={
      title={Recent developments in algebraic topology},
   },
   book={
      series={Contemp. Math.},
      volume={407},
      publisher={Amer. Math. Soc., Providence, RI},
   },
   date={2006},
   pages={135--145},
   review={\MR{2248975 (2007d:19001)}},
   doi={10.1090/conm/407/07674},
}



\bib{LaYa}{article}{
   author={Lawson, H. Blaine, Jr.},
   author={Yau, Shing Tung},
   title={Compact manifolds of nonpositive curvature},
   journal={J. Differential Geometry},
   volume={7},
   date={1972},
   pages={211--228},
   issn={0022-040X},
   review={\MR{0334083}},
}



\bib{OsaPry}{article}{
    title     ={Classifying spaces for families of subgroups for systolic groups},
    author    ={Osajda,Damian},
    author    ={Prytu{\l}a, Tomasz},
    status    ={preprint},
    date      ={2016},
    eprint    ={https://arxiv.org/abs/1604.08478},
}


\bib{Pet}{article}{
   author={Petrosyan, Nansen},
   title={Jumps in cohomology and free group actions},
   journal={J. Pure Appl. Algebra},
   volume={210},
   date={2007},
   number={3},
   pages={695--703},
   issn={0022-4049},
   review={\MR{2324601}},
   doi={10.1016/j.jpaa.2006.11.011},
}

\bib{Seg}{book}{
   author={Segal, Daniel},
   title={Polycyclic groups},
   series={Cambridge Tracts in Mathematics},
   volume={82},
   publisher={Cambridge University Press, Cambridge},
   date={1983},
   pages={xiv+289},
   isbn={0-521-24146-4},
   review={\MR{713786}},
   doi={10.1017/CBO9780511565953},
}



\end{biblist}
\end{bibdiv}
\end{document}